\documentclass[graybox]{SNmult}

\usepackage{type1cm}        
\usepackage{makeidx}         
\usepackage{graphicx}        
\usepackage{multicol}        
\usepackage[bottom]{footmisc}

\usepackage{newtxtext}       %
\usepackage[varvw]{newtxmath}       

\makeindex             
\usepackage{booktabs}
\usepackage{siunitx}

\usepackage{amsthm}
\usepackage{amsfonts}
\usepackage{xcolor}
\usepackage{float}
\usepackage{stmaryrd} 
\usepackage{relsize}
\usepackage{setspace}
\usepackage{caption}
\usepackage{svg}
\usepackage{subfig}
\usepackage{bm}
\usepackage{trajan}
\usepackage{algorithm}
\usepackage{algpseudocode}
\usepackage{url}
\newcommand{\myv}[0]{\boldsymbol{v}}
\newcommand{\myu}[0]{\boldsymbol{u}}
\newcommand{\myf}[0]{\boldsymbol{f}}

\newcommand{\myF}[0]{\boldsymbol{F}}

\newcommand{\myx}[0]{\boldsymbol{x}}

\newcommand{\myw}[0]{\boldsymbol{w}}
\renewcommand{\hat}[1]{\widehat{#1}} 
 
\newcommand{\pd}[2]{\frac{\partial#1}{\partial#2}}

\newcommand{\LRp}[1]{\left( #1 \right)} 

\newcommand{\mylcurly}[0]{\{\!\!\{}
\newcommand{\myrcurly}[0]{\}\!\!\}}

\begin{document}
\title*{Entropy correction artificial viscosity for high order DG methods using multiple artificial viscosities}
\author{Raymond Park\orcidID{0009-0003-8097-6109} and\\ Jesse Chan\orcidID{0000-0003-2077-3636}}
\institute{Raymond Park \at Oden Institute for Computational Engineering and Sciences, The University of Texas at Austin Address of Institute, Austin, Texas \\
\email{rjp2498@my.utexas.edu}
\and Jesse Chan \at Oden Institute for Computational Engineering and Sciences, The University of Texas at Austin, Austin, Texas \at Department of Aerospace Engineering and Engineering Mechanics, The University of Texas at Austin, Austin, Texas \\
\email{jesse.chan@oden.utexas.edu}}
%
%
\maketitle
\abstract{Entropy stable discontinuous Galerkin (DG) methods display improved robustness
for problems with shocks, turbulence, and under-resolved features by enforcing
an entropy inequality. Such methods have traditionally relied on entropy conservative (EC) fluxes that are computationally expensive to evaluate. An alternative approach for enforcing an entropy inequality is through a minimally dissipative ``entropy correction" artificial viscosity. We review how to construct such an artificial viscosity formulation and extend this approach to multiple types of viscosity (e.g., viscosity and thermal
diffusivity). We determine simple analytical expressions for optimal viscosity parameters. We compare this to the case of a single monolithic viscosity parameter for different 1D and 2D problems, and
show that the proposed method allows users to more precisely target specific physical
phenomena while retaining robustness for general problem settings.
}

\section{Introduction}
\label{sec:1}
In the field of computational fluid dynamics (CFD), there has been increasing need for accurate numerical simulations of fluid flow for a variety of application. Finite volume methods (FVM) and finite element methods (FEM) have traditionally served as the main workhorse methods for CFD \cite{ZJWang}. However, certain modern engineering simulations increasingly demand high-fidelity solutions with accuracy levels that are computationally infeasible for lower (first or second order) methods. High order discontinuous Galerkin (DG) methods borrow ideas from FVM and FEM to retain geometric flexibility while achieving arbitrary high order accuracy and are commonly used to approximate solutions to time-dependent nonlinear conservation laws.  While high order methods are efficient and accurate by introducing small dissipation, they still face major challenges related to entropy stability when approximating sharp gradients and under-resolved solution features.

Entropy-stable discontinuous Galerkin methods improve the robustness of high order DG methods without introduction of heuristic parameters \cite{carpenter2014entropy, chen2017entropy, gassner2016split}. They are obtained by combining the theory of entropy stability with classical discontinuous Galerkin. These formulations employ collocated Gauss–Lobatto quadrature nodes, which yield mass and stiffness matrices satisfying the summation-by-parts (SBP) property. Within the flux-differencing framework, the volume integral is then rewritten as an equivalent finite-volume–type differencing formulation using SBP operators together with two-point numerical fluxes that are consistent, symmetric, and entropy conservative. 

Recently, entropy correction artificial viscosity (ECAV) methods were introduced in \cite{chan_av, christner2025entropystablefinitedifference, christner2025entropystablenodaldiscontinuous} as an alternative to flux differencing approaches for enforcing an entropy inequality. These approaches are different from traditional artificial viscosity approach in that the goal is to enforce an entropy inequality as opposed to perform shock capturing. The ECAV coefficient can be computed locally per element basis by computing a local entropy violation. ECAV coefficients are small in magnitude and typically scale quadratically with the approximation error \cite{vanfleet2026viscousdiscretization2026}, which does not significantly restrict the time-step condition.

In this work, we present an extension to the ECAV framework to artificial viscosities with multiple diffusive terms whose scaling parameters we determine in an optimal fashion. This has the potential to tailor an entropy stable artificial viscosity to specific classes of physical phenomena (i.e., overheating, turbulence, shocks, etc) in order to minimize the loss of accuracy and preserve resolution of fine-scale features. By utilizing different viscosity models that target specific physical phenomena (i.e., regions with high temperature gradients or under-resolved turbulence), the proposed method can be tailored to specific classes of problems without introducing heuristic stabilization parameters. 

The paper is organized as follows. Section \ref{sec:nonlinear/entropy} reviews nonlinear conservation laws and various formulations of entropy inequalities, and Section \ref{sec:high_order_dg} reviews a standard DG weak formulation for nonlinear conservation laws and presents a standard viscous discretization to compute the artificial viscosity coefficient for a monolithic viscosity model. In Section \ref{:sec:multiple_av}, we formulate an optimization problem to combine multiple artificial viscosity models and briefly describe the Navier–Stokes diffusion model introduced by Svärd in \cite{svard2024convergent}, as well as a spectral vanishing viscosity (SVV) introduced by Tadmor \cite{tadmor1989convergence}. We conclude with numerical experiments to demonstrate high order accuracy and robustness of the proposed approach.

\section{Nonlinear conservation laws and entropy inequalities}
\label{sec:nonlinear/entropy}
In this paper, our main focus is to numerically approximate the following system of nonlinear equations in $d$ spatial dimensions
\begin{align}
    \frac{\partial \boldsymbol{u}}{\partial t} + \sum_{i = 1}^d \frac{\partial 
    \bm{f}_i(\boldsymbol{u})}{\partial \myx_i} = \boldsymbol{0}.
    \label{conserr_law}
\end{align}
$\boldsymbol{f}_m : \mathbb{R}^n \mapsto \mathbb{R}^n, \boldsymbol{u}(x, t) \in \mathbb{R}^n$ denotes flux functions, and conservative variables, respectively. We assume that $\eqref{conserr_law}$ admits one or more entropy inequalities of the form 
\begin{align}
    \frac{\partial S(\myu)}{\partial t} + \sum_{m = 1}^d \frac{\partial F_m(\myu)}{\partial x_m} \le 0,
    \label{entropy_inequality}
\end{align}
where $S(\myu)$ is a scalar convex entropy function. 
Furthermore, entropy flux can be written as the following
\begin{align}
     \quad F_m(\myu) = \myv(\myu)^T \bm{f}_m(\myu) - \psi_m(\myu),
\end{align}
where $F_m(\myu)$ are the associated entropy fluxes, $\psi_m$ are the entropy potentials, and $\myv(\myu)$ are the entropy variables defined as 
\begin{align}
    \myv(\myu) = \frac{\partial S}{\partial \myu}.
    \label{entropy_flux_relation}
\end{align}
The convexity of entropy function necessitates that the mapping between conservative and entropy variables is invertible. 
Sufficiently smooth numerical solution will satisfy the equality version of $\eqref{entropy_inequality}$ by multiplying \eqref{conserr_law} with entropy variables, using the chain rule, and the following compatibility condition for entropy and entropy fluxes 
\begin{align}
    \bm{v}^T \frac{\partial \bm{f}_m}{\partial \myu} = \frac{\partial F_m}{\partial \myu}, \quad m = 1, ... , d.
\end{align}
In this work, we utilize an integrated cell version of \eqref{entropy_inequality} by considering a closed domain $D \subset \mathbb{R}^d$. We integrate \eqref{entropy_inequality} over cell $D$, and using identity $\eqref{entropy_flux_relation}$ and divergence theorem yields 
\begin{align}
    \int_D \frac{\partial S(\myu)}{\partial t} + \int_{\partial D} \sum_{m = 1}^d \big(\myv^T\myf_m(\myu) - \psi_m(\myu)\big) n_m \le 0,
    \label{cell_entropy_inequality}
\end{align}
where $n_m$ represents the outward normal. Rather than satisfying $\eqref{cell_entropy_inequality}$, we will enforce an intermediate cell entropy inequality below:
\begin{align}
    \sum_{m = 1}^d\bigg[ \int _D -\frac{\myv(\myu)^T}{\partial x_m} + \int_{\partial D} \psi_m(\myu) n_m \bigg] \ge 0.
    \label{final_entropy_inequality}
\end{align}
In this work, we rely on various artificial viscosity models to satisfy a discrete version of $\eqref{final_entropy_inequality}$. For sufficiently smooth $\myu$, and convective flux $\bm{f}_m(\myu)$, $\eqref{final_entropy_inequality}$ is already satisfied.
\section{High order discontinuous Galerkin formulations and viscosity discretizations}
\label{sec:high_order_dg}
\subsection{High order DG formulation}
We assume that the domain $\Omega \subset \mathbb{R}^d$ is triangulated by non-overlapping elements $D^k$, where each element $D^k$ is the image of a reference element $\hat{D}$ under some affine mapping $\phi^k: \hat{D} \rightarrow D^k$. Let $\bm{n}$ denote the outward normal vector on each face of $D^k$. We denote $(u, v)_{D^k}, \langle u, v \rangle_{\partial D^k}$ as the $L^2$ inner products on element $D^k$ and the corresponding surface $\partial D^k$ 
\begin{align*}
    (u, v)_{D^k} \approx \int_{D^k} u(\myx)v(\myx)\, dx, \quad 
        \langle u, v \rangle_{\partial D^k} \approx \int_{\partial D^k} u(\myx) v(\myx)\, dx,
\end{align*}
which can be approximated by using quadrature. The standard weak DG formulation of \eqref{conserr_law} for an approximate solution $\myu_h$ on a single element $D^k$ is the following 
\begin{align}
     \bigg(\myw, \frac{\partial \myu_h}{\partial t}\bigg)_{D^k} + \sum_{i =1}^d \bigg(-\myf_i(\myu_h), \frac{\partial \myw}{\partial x_i} \bigg)_{D^k} + \langle \myf_*, \myw \rangle_{\partial D^k} = \boldsymbol{0}, \quad \bm{w} \in [P^N(D^k)]^n,
     \label{DG}
\end{align}
where $P^N(D^k)$ denotes the space of degree $N$ polynomials on $D^k$. \eqref{DG} is derived by multiplying \eqref{conserr_law} by a test function $\bm{w}$, integrating by parts, and introducing numerical flux $\bm{f}^*_n$ as the interface surface flux. Both volume and surface integrals are approximated using quadrature. The global DG formulation is derived by summing up over all elements $D^k$.

In this work, we utilize the DG spectral element method (DGSEM) which uses Legendre-Gauss-Lobotto (LGL) quadrature points that are collocated with interpolation points. It can be shown that, under collocation at LGL nodal points, the mass matrix and differentiation matrix satisfy the summation-by-parts (SBP) property, which is a discrete analogue of integration by parts \cite{gassner2013skew}.
\subsection{Entropy correction artificial viscosity}
In this section, we briefly review our viscous discretization, and how to compute artificial viscosity coefficient to enforce the intermediate entropy \eqref{final_entropy_inequality}.  We consider the modified nonlinear conservation law by introducing the following Laplacian type artificial viscosity term to \eqref{conserr_law}
\begin{align}
     \frac{\partial \boldsymbol{u}}{\partial t} + \sum_{i = 1}^d \frac{\partial \bm{f}_i(\boldsymbol{u})}{\partial x_i} = \sum_{i, j = 1}^d\frac{\partial}{\partial x_i} \bigg(\epsilon_k(\myu) \frac{\partial \myu}{\partial x_j} \bigg), 
     \label{eq:conser_law_av}
\end{align}
where $\epsilon_k(\myu) \ge 0$ is the artificial viscosity coefficient to be added to enforce the entropy inequality. We assume that $\epsilon_k(\myu)$ is constant over element $D^k$. A symmetric form of \eqref{eq:conser_law_av} can be derived using the chain rule, i.e., $\frac{\partial \myu}{\partial x} = \frac{\partial \myu}{\partial \myv} \frac{\partial \myv}{\partial x}$, where we write the artificial viscosity in terms of the entropy variables:
\begin{align}
     \frac{\partial \boldsymbol{u}}{\partial t} + \sum_{i = 1}^d \frac{\partial \myf_i(\boldsymbol{u})}{\partial x_i} = \sum_{i = 1}^d\frac{\partial}{\partial x_i} \bigg(\epsilon_k(\myu) \frac{\partial \myu}{\partial \myv}\frac{\partial \myv}{\partial x_j} \bigg).
     \label{AV}
\end{align}
A more general viscous regularization term will be analyzed
\begin{equation}
\pd{\bm{u}}{t} + \sum_{i=1}^d \pd{\bm{f}_i(\bm{u})}{x_i} = \sum_{i,j=1}^d \pd{}{x_i}\LRp{\epsilon_k(\bm{u})\bm{K}_{ij}\pd{\bm{v}}{x_j}}.
\label{eq:av}
\end{equation}
Here, $\bm{K}_{ij}$ denote blocks of a symmetric and positive semi-definite matrix $\bm{K}$
\begin{equation}
\bm{K} = \begin{bmatrix}
\bm{K}_{11} & \ldots & \bm{K}_{1d}\\
\vdots & \ddots & \vdots\\
\bm{K}_{d1} & \ldots & \bm{K}_{dd}
\end{bmatrix} = \bm{K}^T, \qquad \bm{K} \succeq 0.
\label{eq:K}
\end{equation}
Note that taking $\bm{K}_{ij} = \delta_{ij} \pd{\bm{u}}{\bm{v}}$ recovers the conservation law with Laplacian-based artificial viscosity \eqref{AV}. 

For the $d$-dimensional compressible Euler equations governed by the ideal gas law, explicit expression of the Jacobian $\frac{\partial \myu}{\partial \myv}$ exist. For example, for the 2 dimensional case, such expressions can be found in \cite{barth1999numerical, chan_av, godlewski2013numerical}.

\subsection{DG formulation of monolithic ECAV model}
Denote the viscous term as $\boldsymbol{g_\text{visc}}$, and $\myu_h$ as the approximate solution. Consider the weak form of \eqref{AV}, 
\begin{align}
     \bigg(\myw, \frac{\partial \myu_h}{\partial t}\bigg)_{D^k} + \sum_{i =1}^d \bigg(-\myf_i(\myu_h), \frac{\partial \myw}{\partial x_i} \bigg)_{D^k} + \langle \myf^*_n, \myw \rangle_{\partial D^k} = (\boldsymbol{g_\text{visc}}, \myw)_{D^k}, \quad \myw \in [P^N(D^k)]^n
     \label{eq:DG_AV}
\end{align}
where $\bm{f}^*_n$ is the surface flux. For example, $\bm{g}_\text{visc} \approx \epsilon_k \Delta \myu$ if the viscous term is the Laplacian-based artificial viscosity where $\epsilon_k$ is the viscosity coefficient. The viscous term will be computed using a BR-1 type discretization \cite{bassi1997high}
\begin{align}
    (\bm{\Theta_i}, \myw_{1, i})_{D^k} = \bigg(\frac{\partial \bm{v}_h}{\partial x_i}, \bm{w}_{1, i} \bigg)_{D^k} + \frac{1}{2} \langle \llbracket \myv_h \rrbracket n_i, \myw_{1, i} \rangle_{\partial D^k}, \quad \forall \myw_{1, i} \in [P^N(D^k)]^n 
    \label{DGofv} \\ 
    (\boldsymbol{\sigma}_i, \myw_{2, i})_{D^k} = \Bigg( \sum_{j = 1}^d \epsilon_k(\myu_h) \boldsymbol{K}_{ij} \bm{\Theta}_j, \myw_{2, i}\Bigg)_{D^k}, \quad \forall \myw_{2, i} \in [P^N(D^k)]^n 
    \label{L2oftheta_sigma}
    \\
    (\boldsymbol{g_\text{visc}}, \myw_3)_{D^k} = \sum_{i =1}^d \bigg[ \bigg(-\boldsymbol{\sigma}_i, \frac{\partial \myw_3}{\partial x_i}\bigg) + \langle \mylcurly \boldsymbol{\sigma}_i \myrcurly n_i, \myw_3 \rangle_{\partial D^k} \bigg], \quad \forall \myw_3 \in [P^N(D^k)]^n,
    \label{viscous_calc}
\end{align}
where $n_i$ is the $i$th component of the outward normal vector of element $D^k$ in the $i$th Cartesian direction, and $\bm{v}_h$ is an approximation to $\bm{v}(\myu_h)$ which we will describe in more detail in the following paragraph. We introduce the jumps and average notations
\begin{align*}
    \llbracket u \rrbracket=  u^+ - u^-, \quad \mylcurly u \myrcurly = \frac{1}{2} (u^+ + u^-).
\end{align*}
Here, $u^+, u^-$ represent exterior and interior values of $u$ on a given element $D^k$. Equation \eqref{DGofv} calculates $\bm{\Theta}_i$ which represents the DG derivative with respect to $x_i$ (the $i$th coordinate direction) of the entropy variables $\myv(\myu)$. Equation \eqref{viscous_calc} computes the divergence of the viscous fluxes calculated from the above using integration by parts twice, and DG approximation. Equation \eqref{L2oftheta_sigma} computes the $L^2$ projection of the viscous flux $\sum_{i =1}^d \boldsymbol{K}_{ij} \bm{\Theta}_j$ onto the space of degree $N$ polynomials $P^N$. 

We note that, to ensure $\bm{g}_{\text{visc}}$ is discretely entropy dissipative, $\bm{v}_h$ must be taken to be the $L^2$ projection of the entropy variables. For the nodal collocation schemes used in the numerical experiments of this work, the calculation of $\bm{v}_h$ reduce to the interpolant of $\bm{v}(\myu_h)$ at collocation points. More details on this can be found in $\cite{chan_av, chan2018discretely}$.
\begin{lemma}
    Let $\epsilon_k(\myu_h) \ge 0$, and $\bm{g}_\text{visc}$ be given by \eqref{DGofv}, \eqref{L2oftheta_sigma}, \eqref{viscous_calc}. Then, for a periodic domain, 
    \begin{align}
        \sum_k -(\bm{g}_\text{visc}, \myv_h)_{D^k} = \sum_k \sum_{i, j = 1}^d (\epsilon_k(\myu_h)\bm{K}_{ij} \bm{\Theta}_j, \bm{\Theta}_i)_{D^k} \ge 0
    \end{align}
\label{lemma:positive definite of av}
\end{lemma}
The proof can be found in \cite{chan2022entropy} under Lemma 3.1. Lemma \ref{lemma:positive definite of av} indicates that, using a BR1-type discretization, Laplacian artificial viscosity is provably dissipative provided that the viscosity coefficient is nonnegative. Next, we define the entropy volume residual as:
\begin{align}
    \delta_k(\myu_h) = \sum_{m = 1}^d \int_{D^k} -\bigg(\frac{\partial \myv_h}{\partial x_m}\bigg)^T \myf_m(\myu_h) + \int_{\partial D^k}  \psi_i(\tilde{\myu}_h) n_m,
    \label{entropy_volume_residual}
\end{align}
where $\tilde{\bm{u}}_h$ is the conservative variable mapped from the projected entropy variables. Comparing this $\delta_k(\myu_h)$ with \eqref{final_entropy_inequality}, we see that $\delta_k(\myu_h)$ must be non-negative to satisfy the entropy inequality. We compute the minimum $\epsilon_k(\myu_h)$ such that the following lemma holds.
\begin{lemma}
    \label{lemma:eps}
    Let $\epsilon_k(\myu_h)$ on $D^k$ satisfy 
    \begin{align*}
        \sum_{i, j = 1}^d \epsilon_k(\myu_h)(\boldsymbol{K}_{ij}\boldsymbol{\Theta}_j, \boldsymbol{\Theta}_i)_{D^k} \ge -\min(0, \delta_k(\myu_h)),
    \end{align*}
then the solution to \eqref{eq:DG_AV} $\myu_h$ satisfies the global entropy inequality
\begin{align*}
    \sum_k \bigg[\bigg(\frac{\partial S (\myu_h)}{\partial t}, 1 \bigg)_{D^k} + \bigg \langle \myv_h^T \myf_n^* - \sum_{i = 1}^d \psi_i(\tilde{\myu}_h)n_i, 1 \bigg \rangle_{\partial D^k}  \bigg] \le 0.
\end{align*}
\label{lemma:entropy stable av}
\end{lemma}
The proof can be found in \cite{chan_av} under Lemma 2. For the monolithic type viscosity model, artificial viscosity coefficient can be computed by solving $\epsilon_k(\myu_h)$ from Lemma \ref{lemma:eps}. By assuming that $\epsilon_k(\bm{u}_h)$ is constant over each element, the minimum value of $\epsilon_k(\myu_h)$ to enforce the entropy inequality \eqref{final_entropy_inequality} is defined as follows:
\begin{align}
    \epsilon_k(\myu_h) \ge \frac{-\min(0, \sigma_k(\myu_h))}{\sum_{i, j=1}^d (\bm{K}_{ij}\bm{\Theta}_j, \bm{\Theta}_i)}
    \label{eq:coef_form}
\end{align}
The denominator of the right hand side of $\eqref{eq:coef_form}$ can be arbitrarily close to $0$. Therefore, we use a regularized ratio by computing $\frac{a}{b} \approx \frac{ab}{\delta + b^2}$ using a small tolerance $\delta = 10^{-14}$ to compute $\epsilon_k(\myu_h)$. 

Different discretizations of the viscous term can yield different theoretical and numerical behavior. For example, it was shown in \cite{vanfleet2026viscousdiscretization2026} that the local discontinuous Galerkin (LDG) discretization admits an $O(h)$ upper bound on the artificial viscosity coefficient, which does not adversely impact the time-step constraint. However, in practice, we do not typically observe a significant difference between LDG and BR1 when using an upwind-like interface flux for the convective term \cite{vanfleet2026viscousdiscretization2026}.

\section{Multiple artificial viscosities approach}
We now present the main contribution of this paper: incorporation of different viscosity models in an optimal parameter free fashion. 
\label{:sec:multiple_av}
\subsection{DG Formulation of multiple AV models}
Let $M$ be the total number of viscous terms added to the conservation law \eqref{conserr_law}
\begin{align}
     \frac{\partial \boldsymbol{u}}{\partial t} + \sum_{i = 1}^d \frac{\partial \bm{f}_i(\boldsymbol{u})}{\partial x_i} = \sum_{m=1}^M \bm{g}_{\text{visc}}^m
     \label{multiple_av_conser_law},
\end{align}
where $m = 1, ...,M$ indexes the viscosity model. $\bm{g}_{\text{visc}}^m$ represents different viscous terms. The coefficient of each viscosity model is computed over each element $D^k$. Each coefficient can vary independently and is taken to be constant over each element.
The weak form of \eqref{multiple_av_conser_law} is the following:
\begin{align}
     \bigg(\myw, \frac{\partial \myu_h}{\partial t}\bigg)_{D^k} + \sum_{i =1}^d \bigg(-\myf_i(\myu_h), \frac{\partial \myw}{\partial x_i} \bigg)_{D^k} + \langle \myf_*, \myw \rangle_{\partial D^k} = \sum_{m = 1}^M (\boldsymbol{g}_{\text{visc}}^m, \myw)_{D^k}, \quad \myw \in [P^N(D^k)]^n.
     \label{DG_visc}
\end{align}
In this work, each viscous term will be computed using the BR-1 type discretization. 
\subsection{Optimization problem involving the cell entropy inequality}
There now exist $M$ artificial viscosity coefficients $\epsilon^m_k(\myu_h), m = 1, ..., M$ that must be computed. Similar to \cite{chan_av}, we desire to compute the minimum amount of AV to satisfy the cell entropy inequality
\begin{align}
    \sum_{i = 1}^d \bigg(-\frac{\partial \myv_h}{\partial x_i} ,\myf_i(\myu_h)\bigg)_{D^k} + \langle\psi_i(\myu_h)n_i, 1\rangle_{\partial D^k} +\sum_{m=1}^M \Big(\bm{g}_{\text{visc}}^m, \myv\Big)_{D^k}\ge 0
    \label{eq:cont_cell_entropy_inequality_multi_av}.
\end{align}
Introducing too much dissipation can lead to overly diffused solution. Thus, we wish to determine $\epsilon^m_k(\myu_h)$ such that the inequality \eqref{eq:cont_cell_entropy_inequality_multi_av} becomes an equality. We can then construct a minimization problem over $\bm{\epsilon}$ with a simple closed form solution:
\begin{align}
    &\min \frac{1}{2} \boldsymbol{\epsilon}^T\boldsymbol{W} \boldsymbol{\epsilon} \nonumber \\
    s.t \quad \boldsymbol{r}^T \boldsymbol{\epsilon} &= \sigma_k(\myu_h) 
    \label{eq:linear_constraint}\\
    \boldsymbol{\epsilon} &= 
    \begin{bmatrix}
        \epsilon_1, & \cdots & ,\epsilon_M
    \end{bmatrix} \nonumber \\
    \boldsymbol{r}_m &= \sum_{i, j =1}^d(\boldsymbol{K}_{ij}^m \boldsymbol{\Theta}_{m, j}, \boldsymbol{\Theta}_{m, i})_{D^k}, \quad \forall m \in 1, ..., M
    \label{optimization_prob}
\end{align}
where $\boldsymbol{W}$ is a diagonal matrix. For each element $D^k$, \eqref{optimization_prob} solves for $\bm{\epsilon}_m, \forall m = 1, ..., M$.\ \eqref{optimization_prob} is a quadratic optimization problem with a linear constraint which has a simple closed solution using the method of Lagrange multipliers. 

\begin{remark}
    To solve the optimization problem \eqref{optimization_prob}, the Karush-Kuhn-Tucker (KKT) matrix can be formed to solve for the optimal solution
\begin{align*}
    \begin{bmatrix}
        \boldsymbol{W} & \boldsymbol{r^T} \\
        \boldsymbol{r}  & \boldsymbol{0} 
    \end{bmatrix}
    \begin{bmatrix}
        \boldsymbol{\epsilon}^* \\
        \delta^*
    \end{bmatrix} = 
    \begin{bmatrix}
        \boldsymbol{0} \\
        \sigma_k(\myu_h)
    \end{bmatrix},
\end{align*}
where $\delta^*$ is a scalar, and $\bm{\epsilon^*}$ is the solution to \eqref{optimization_prob}. However, we resort to showing a closed form of the solution without inverting the KKT matrix.
\end{remark}
We define the Lagrangian of \eqref{optimization_prob}
\begin{align*}
    \mathcal{L}(\boldsymbol{\epsilon}, \lambda) = \frac{1}{2}\boldsymbol{\epsilon}^T\boldsymbol{W} \boldsymbol{\epsilon}  + \lambda (\boldsymbol{r}^T \boldsymbol{\epsilon} - \sigma_k(\myu_h)),
\end{align*}
where $\lambda$ is a scalar quantity. Define the vector $\boldsymbol{x} = [\boldsymbol{\epsilon} \hspace{5pt} \lambda]^T$, then
\begin{align*}
    \nabla_{\myx} \mathcal{L}(\myx) = 
    \begin{bmatrix}
        \boldsymbol{\epsilon}^T \boldsymbol{W} + \lambda \boldsymbol{r} \\ \boldsymbol{r}^T \boldsymbol{\epsilon} - \sigma_k(\myu_h)
    \end{bmatrix}.
\end{align*}
The solution to \eqref{optimization_prob} is the value at which $\bm{x}$ for which the gradient of the Lagrangian is $\bm{0}$
\begin{align*}
    \begin{bmatrix}
        \boldsymbol{\epsilon}^T \boldsymbol{W} + \lambda \boldsymbol{r} \\\boldsymbol{r}^T \boldsymbol{\epsilon} - \sigma_k(\myu_h)
    \end{bmatrix} = \boldsymbol{0}.
\end{align*}
Note that the last equation is equivalent to the imposed linear constraint. The method of Lagrange multipliers is well established, so we omit the details for brevity. For the present work, the weight matrix $\bm{W}$ is chosen as the identity, implying no preferential weighting among the viscosity models. Under this assumption, the closed form solution of $\bm{\epsilon}$ is
\begin{align}
    \bm{\epsilon}_i = \frac{\sigma_k(\myu_h)\bm{r}_i}{\sum_{j =1}^M\bm{r}_j^2}.
\end{align}
When $M = 2$, we can define $a, b$ as \newline $\bm{r}_1 = \sum_{i, j}^d(\bm{K}_{1, ij} \bm{\Theta}_{1, j}, \bm{\Theta}_{1, i})_{D^k}, \bm{r}_2 = \sum_{i, j}^d(\bm{K}_{2, ij} \bm{\Theta}_{2, j}, \bm{\Theta}_{2, i})_{D^k}$, respectively, then 
\begin{align}
    \boldsymbol{\epsilon}_{1} = \frac{\sigma_k(\myu_h)a}{a^2 + b^2}, \quad \bm{\epsilon}_{2} = \frac{\sigma_k(\myu_h)b}{a^2 + b^2}.
    \label{closed_form}
\end{align}
The closed form \eqref{closed_form} avoids the need to invert the KKT matrix.
\subsection{Viscosity models}
In this section, we briefly introduce two viscosity models used in this work.
\subsubsection{Sv\"{a}rd thermal viscosity}
Using the proposed optimization problem, we apply it to the modified formulation of the compressible Navier-Stokes equations proposed by Sv\"{a}rd \cite{svard2024convergent}: 
\begin{align*}
    \frac{\partial \myu}{\partial t} + \sum_{i = 1}^3\frac{\partial \myf_i(\myu)}{\partial x_i} = \epsilon^1 \Delta \myu + \begin{bmatrix}
        0 \\ 0 \\ \nabla \cdot (\epsilon^2 \nabla T)
    \end{bmatrix}
\end{align*}
where $\epsilon^2, T$ denotes the thermal diffusivity, and temperature, respectively. Compared with the Laplacian artificial viscosity formulation \eqref{eq:conser_law_av}, the Sv\"{a}rd formulation includes an additional thermal diffusion term. The Laplacian and thermal component will be considered as the first and second viscosity model, respectively. We allow the diffusive coefficients $\epsilon^1, \epsilon^2$ to vary independently. However, one needs to derive the corresponding entropy variables of the specific energy $E$, and the term $\frac{dE}{dv_E}$, where $v_E = \frac{\partial S}{\partial E}$ represents the entropy variable corresponding to the total energy variable. 
The term $\frac{dE}{dv_E}$ was found to be
\begin{align}
    \frac{dE}{dv_E} = \sum_{i=1}^d (\gamma - 1)^2 \Bigg(\frac{(E - \frac{1}{2} \rho u_i^2)}{\rho}\Bigg)^2 > 0.
    \label{eq:dEdvE}
\end{align} 
Note that this term is strictly non-negative. The temperature $T$ is a scalar quantity which simplifies this term to a scalar quantity as well. The thermal component of Sv\"{a}rd viscosity can be discretized using the BR1 discretizations. Then, the entropy contribution of the Sv\"{a}rd thermal viscosity is computed as
\begin{align*}
    (\bm{g_\text{visc}}^{\text{Sv\"{a}rd}}, \myv) 
    &= (\bm{g}_\text{visc}^{1}, \myv) + (\bm{g}_{\text{visc}}^2, \myv) \\
    &=-\bigg[\epsilon_k^1(\myu_h)\Big((\bm{K}_{1, ij} \bm{\Theta}_{1, j}, \bm{\Theta}_{1, i})_{D^k} + \epsilon_k^2(\myu_h)(\bm{K}_{2, ij} \bm{\Theta}_{2, j}, \bm{\Theta}_{2, i})_{D^k}\bigg].
\end{align*}
$\bm{g}_\text{visc}^1, \bm{g}_\text{visc}^2$ are the Laplacian, and thermal viscous terms, respectively, and $\bm{K}_{1, ij}, \bm{K}_{2, ij}$ are defined as
\begin{align}
    \bm{K}_{1, ij} = \delta_{ij}\frac{\partial \myu}{\partial \myv}, \quad \bm{K}_{2, ij} = \delta_{ij}\begin{bmatrix}
        0 & 0 & 0 \\
        0 & 0 & 0 \\
        0 & 0 & \frac{\partial E}{\partial v_E}
    \end{bmatrix}.
    \label{eq:K2_defintion}
\end{align}
and $\bm{\Theta}_{2, i}, \bm{\sigma}_{2, i}$ is the the DG approximation of $v_E$ padded with zeros and $L^2$ projection of $\sum_{i, j = 1}^d \bm{K}_{2, ij}\bm{\Theta}_{2, j}$, respectively
\begin{align}
    \bm{\Theta}_i \approx 
    \begin{bmatrix}
        0 \\
        0 \\
        \frac{\partial v_E}{\partial x_i}
    \end{bmatrix}, \quad (\bm{\sigma}_{2, i}, \myw)_{D^k} = \bigg(\sum_{i, j = 1}^d \bm{K}_{2, ij}\bm{\Theta}_{2, j}, \myw\bigg), \quad \forall \myw \in [P^N(D^k)]^n
    \label{eq:DG_thermal_defintion}.
\end{align}
\begin{remark}
    The BR-1 discretization of the thermal component of the Sv\"{a}rd viscosity model is also provably dissipative. The proof relies on the sign of the term $\frac{dE}{dv_E}$ which is non-negative. We first substitute $\bm{\Theta}_{2, i}$, $\bm{K}_{2, ij}$ defined in \eqref{eq:K2_defintion}, \eqref{eq:DG_thermal_defintion}, $L^2$ projection of $\sum_{i, j = 1}^d \bm{K}_{2, ij}\bm{\Theta}_{2, j}$ for $\bm{\Theta}_i, \bm{K}_{ij}, \bm{\sigma}_i$ in Lemma 3.1 of \cite{chan2022entropy}, respectively.  The boundary terms also disappear using Lemma \ref{lemma:positive definite of av} by letting $\bm{w}_{2, i} = \bm{\Theta}_{2, i}, \bm{w}_{1, i} = \bm{\sigma}_{2, i}$. Summing up over all elements yields that $\sum_k -(\bm{g_\text{visc}}^2, \myv) \ge 0.$
\label{remark:svard positive av}
\end{remark}
We utilize the aforementioned discretization of the Laplacian viscosity model for the Sv\"{a}rd viscosity model, and Lemma \ref{lemma:positive definite of av}, Remark \ref{remark:svard positive av} indicate that the entropy contribution of each of the viscous terms (Laplacian, thermal) are positive. Then, we can extend Lemma \ref{lemma:entropy stable av} by summing over the two viscous terms 
\begin{align*}
    -(\bm{g}_{\text{visc}}^{\text{Sv\"{a}rd}}, \myv) 
    &= -\sum_{m = 1}^{2} (\bm{g}_\text{visc}^m, \myv)  \\
    &= \sum_{i = 1}^d \bigg[ \epsilon_k^1(\myu_h)(\bm{\Theta}_{1, i}, \bm{K}_{1, ii} \bm{\Theta}_{1, i} )_{D^k} + \epsilon_k^2(\myu_h) \Big (\bm{\Theta}_{2, i}, \bm{K}_{2, ii}\bm{\Theta}_{2, i} \Big)_{D^k}\bigg] \ge 0
\end{align*}

\begin{remark}
    From Lemma \ref{lemma:positive definite of av} and Remark \ref{remark:svard positive av}, both the Laplacian and thermal viscous discretizations are provably entropy dissipative. Thus, the optimization problem \eqref{optimization_prob} directly solves for the condition under Lemma $\ref{lemma:entropy stable av}$, and the computed $\epsilon_k^m(\myu_h)$ enforces the global entropy inequality. 
\end{remark}
\subsubsection{Spectral vanishing viscosity (SVV)}
Spectral Vanishing Viscosity (SVV) was first introduced by Tadmor \cite{tadmor1989convergence} where it was applied to Burger's equation, i.e.,
\begin{align*}
    \frac{\partial u}{\partial t} + \frac{\partial u^2/2}{\partial t} = \epsilon \Big( \frac{\partial}{\partial x} \bm{\mathcal{F}}  \frac{\partial u}{\partial x} \Big),
\end{align*}
where $\bm{\mathcal{F}} = \bm{VF}\bm{V}^{-1}$ denotes the filtering operator, $\bm{V}_{ij} = \phi_j(x_i)$ denotes the generalized Vandermonde matrix, $\phi_j$ denotes the $j$th orthonormal polynomial, $x_i$ are interpolation points, and $\bm{F}$ is a diagonal matrix. The filtering operator $\bm{\mathcal{F}}$ acts on the modal space of the solution by using a high-pass filter, only affecting the high modes and not modifying the lower modes. 

Tadmor showed that this approach damps spurious oscillations, and converges to the unique entropy solution for the inviscid Burger's equation \cite{tadmor1989convergence}. Since then, SVV has gained a lot of traction, especially for 3D turbulence modeling, Large Eddy Simulation (LES), and high Reynolds flow as the SVV coefficient vanishes in the laminar limit, and stabilizes the solution with turbulent flows \ \cite{manzanero2020design, ferrer2017interior, moura2016eigensolution, karamanos2000spectral}. Though the present work only deals with simulation up to two dimensions, we expect similar behavior when simulating flows with high shear stress and vortices. 

What remains is to decide what filter kernel to use. Tadmor introduced a ``step" high filter pass \cite{tadmor1989convergence, mateo2022entropy}
\begin{align*}
\bm{F}_{k} = 
    \begin{cases} 
      0 & k\leq M \\
     1 & k > M
   \end{cases}\hspace{5pt}.
\end{align*}
Another choice filter uses an exponentially decaying coefficients \cite{mateo2022entropy, maday1993legendre} 
\begin{align*}
\bm{F}_{k} = 
    \begin{cases} 
      0 & k\leq M \\
     \exp\Big[-\big(\frac{k - N}{k - M}\big)^2\Big] & k > M,
   \end{cases}
\end{align*}
where $k, M, N$ denotes the wavenumber (polynomial degree), cutoff mode, degree of approximation, respectively. Another choice of filter kernel is\ \cite{mateo2022entropy}
\begin{align*}
    \bm{F}_{i} = \bigg(\frac{i - 1}{N}\bigg)^2.
\end{align*}
It is clear how the filtering operator will act on 1D solutions. For multidimensional systems, we will apply a tensor product of the filtering operator, i.e., $\bm{\mathcal{F}}_{2D} = \bm{\mathcal{F}}_{1D} \otimes \bm{\mathcal{F}}_{1D}$ if assuming the element is quadrilateral. The 2D filter can be interpreted as fixing a filter weight for a given power in one Cartesian direction, and then making adjustments to the weights as the power of the other Cartesian direction is changed. 

Here, we introduce a provably dissipative discretization to incorporate SVV as the secondary viscosity model of choice using BR1. We will incorporate both the SVV model (indexed as $m = 2$) with the Laplacian viscosity model (indexed as $m = 1$). We apply filtering to the DG gradient of $\myv$, and then apply another filtering operation to the $L^2$ projection of $\sum_{i, j=1}^d\bm{K}_{ij}\bm{\mathcal{F}}\bm{\Theta}_j$. Denote $\bm{\mathcal{F}} = \bm{V}\myF\bm{V}^{-1}$, then the BR1 discretization of SVV model with the filtering procedure as described above can be formulated as
\begin{align*}
    (\bm{\Theta}_i, \bm{\sigma}_{2, i})_{D^k} = \bigg(\frac{\partial \bm{v}}{\partial x_i}, \bm{\sigma}_{2, i} \bigg)_{D^k} &+ \frac{1}{2} \langle \llbracket \myv_h \rrbracket n_i, \bm{\sigma}_{2, i} \rangle_{\partial D^k}, \quad \forall \bm{\sigma}_{2, i} \in [P^N(D^k)]^n  \\ 
    \bm{\Theta}_{2, i} &= \bm{\mathcal{F}}\bm{\Theta}_i \\
    (\boldsymbol{\sigma}_i, \bm{\Theta}_i)_{D^k} = \Bigg( \sum_{j = 1}^d \epsilon^{2}_k(\myu_h) &\boldsymbol{K}_{ij} \boldsymbol{\Theta}_{2, j}, \bm{\Theta}_i \Bigg)_{D^k}, \quad \forall \bm{\Theta}_i \in [P^N(D^k)]^n \\
    \bm{\sigma}_{2, i} &= \bm{\mathcal{F}}\bm{\sigma}_i \\
    (\boldsymbol{g_\text{visc}}^{2}, \myv)_{D^k} = \sum_{i =1}^d \bigg[ \bigg(\boldsymbol{-\sigma}_{2, i}, \frac{\partial \myv}{\partial x_i}\bigg)_{D^k} &+ \langle \mylcurly \boldsymbol{\sigma}_{2, i} \myrcurly n_i, \myv \rangle_{\partial D^k} \bigg], \quad \forall \myv \in [P^N(D^k)]^n
\end{align*}
Similar steps to Lemma 3.1 in \cite{chan2022entropy} yield that
\begin{align*}
    -(\bm{g}_\text{visc}^{2}, \myv)_{D^k} = \epsilon_k^{2}(\myu_h)\bigg(\frac{\partial \myv}{\partial x}, \boldsymbol{\mathcal{F}}\bm{K}\bm{\mathcal{F}}\frac{\partial \myv}{\partial x}\bigg)_{D^k}.
\end{align*}
However, it remains unclear if the above quantity is strictly non-negative. Non-negativity follows if $\bm{\mathcal{F}}$ is symmetric with respect to the quadrature-based $L^2$ inner product, i.e.,
\begin{align*}
    (\bm{x}, \bm{\mathcal{F}y})_{D^k} &= (\bm{\mathcal{F}x}, \bm{y})_{D^k} \Rightarrow \\
    \bm{x}^T\bm{MF}\bm{y} &=  \bm{x}^T\bm{F^T}\bm{My} 
\end{align*}
where $\bm{M}$ is the quadrature-based mass matrix.
\begin{lemma}
    Let $\bm{\mathcal{F}} = \bm{VF}\bm{V}^{-1}$ denote the filtering operator, where $\bm{V}_{ij} = \phi_j(x_i)$ denotes the generalized Vandermonde matrix, $\phi$ denotes the $j$th orthonormal polynomial basis function. Here, $\bm{F}$ is a diagonal matrix, and $\bm{M}$ is the mass matrix. Suppose that $\bm{V}^T\bm{MV} = \bm{D}$, where, $
    \bm{D}$ is a diagonal matrix, then $\bm{\mathcal{F}}$ is symmetric respect to the quadrature-based $L^2$ inner product, i.e., 
    \begin{align*}
        \bm{M\mathcal{F}} = (\bm{M\mathcal{F}})^T.
    \end{align*}
    \label{lemma:filter_symmetric_respect_mass}
\end{lemma}
\begin{proof}
    First note the mass matrix $\bm{M}$ is symmetric. We use the following identity
    \begin{align}
        \bm{V}^T\bm{MV} = \bm{D}.
        \label{eq:vandermonde_modal_basic}
    \end{align}
    Two other identities will be used by multiplying \eqref{eq:vandermonde_modal_basic} with $\bm{V}, \bm{V}^{-1}$
    \begin{align*}
        \bm{MV} = \bm{V^{-T}D}, \quad \bm{V^T}\bm{M} = \bm{DV^{-1}}.
    \end{align*}
    Using the above identities, 
    \begin{align*}
        (\bm{M\mathcal{F}})^T  
        &= \bm{V^{-T}F}\underbrace{\bm{{V^T}}\bm{M}}_{\bm{DV^{-1}}} 
        = \underbrace{\bm{V^{-T}D}}_{\bm{MV}}\bm{D^{-1}F}\bm{D}V^{-1} \\
        &= \bm{MV\underbrace{D^{-1}FD}_{\bm{F}}V^{-1}} 
        = \bm{MVFV^{-1}} \\
        &= \bm{M\mathcal{F}}
    \end{align*}
    The third equality holds because $\bm{D}, \bm{F}$ are both diagonal matrices.
\end{proof}
\begin{remark}
    If we compute the inner products using exact integration (i.e Gauss integration), then the following identity holds
    \begin{align*}
        \bm{V^TMV} &= \bm{I}.
    \end{align*}
    where $\bm{I}$ is the identity matrix. However, in this work, we are approximating the inner product using Gauss-Lobatto integration. A different identity must be used
     \begin{align*}
         \bm{V^TMV} = \bm{D}, \quad 
         \bm{D} =
         \begin{cases}
             \delta_{ij} & i + j \neq 2N \\
             \frac{2N + 1}{N} & i = j = N
         \end{cases} 
     \end{align*}
     where $N$ is the degree of approximation. The proof of this identity is in \cite{gassner2011comparison}. For either choice of integration, the assumption that $\bm{V^TMV}$ is a diagonal matrix holds.  
\end{remark}
Since the nodal DGSEM formulation utilizes Legendre-Gauss-Lobotto quadrature points, Lemma \ref{lemma:filter_symmetric_respect_mass} implies that the entropy contribution of SVV is non-negative: 
\begin{align*}
    \bigg(\frac{\partial \myv}{\partial x}, \boldsymbol{\mathcal{F}}\bm{K}\bm{\mathcal{F}}\frac{\partial \myv}{\partial x}\bigg)_{D^k}
    &= \bigg(\bm{\mathcal{F}}\frac{\partial \myv}{\partial x}, \bm{K}\bm{\mathcal{F}}\frac{\partial \myv}{\partial x}\bigg)_{D^k} \ge 0
\end{align*}
which is positive since $\bm{K}$ is a positive definite matrix.
\begin{remark}
    The BR-1 discretization of spectral vanishing viscosity model is also provably dissipative. The proof relies on the positive definiteness of $\bm{K}_{ij}$ and the order of which the filter is applied. We substitute $\bm{\sigma}_{2, j}, \sum_{i, j=1}^d(\bm{K}_{ij}
    \bm{\mathcal{F}}\bm{\Theta}_j, \bm{\mathcal{F}}\bm{\Theta}_i)_{D^k}$ for $\bm{\sigma}_j$ and $\sum_{i, j = 1}^d (\bm{K}_{ij} \bm{\Theta}_j, \bm{\Theta}_i)_{D^k}$ found in Lemma 3.1 of \cite{chan2022entropy}. The boundary terms also disappear using Lemma \ref{lemma:positive definite of av} by letting $\bm{w}_{2, i} = \bm{\Theta}_i, \bm{w}_{1, i} = \bm{\sigma}_{2, i}$. Summing up over all elements yields that that $\sum_k -(\bm{g_\text{visc}}^{2}, v) \ge 0.$
\label{remark:svv positive av}
\end{remark}
\begin{remark}
    From Lemma \ref{lemma:positive definite of av}, and Remark \ref{remark:svv positive av}, both the Laplacian and SVV discretization are provably dissipative. Thus, the optimization problem \eqref{optimization_prob} directly solves for the condition under Lemma $\ref{lemma:entropy stable av}$. Therefore, the computed $\epsilon_k^m(\myu_h)$ satisfy the global entropy inequality. 
\end{remark}
\section{Numerical experiments}
\label{sec:numerical_experiments}
In this section, we present numerical experiments that demonstrate both the robustness of this new approach and verify that the proposed approach improves upon existing methods. Collocation method will be utilized using Legendre Gauss-Lobotto (LGL) nodes. For 2D, we use quadrilateral generated mesh. For the surface the flux, we use local Lax-Fredreichs flux (LLF). $N, M$ represents degree of approximation and number of elements in each Cartesian direction, respectively. The final simulation is denoted as $T$. An adaptive four-stage, third-order strong stability preserving Runge–Kutta time stepper (SSPRK(4,3)) with absolute and relative tolerances of $10^{-6}, 10^{-4}$, respectively, was used. All experiments are implemented in Julia using \texttt{StartUpDG.jl} and \texttt{Trixi.jl} \cite{ranocha2021adaptive}
libraries. For time integration, we utilize the \texttt{OrdinaryDiffEq.jl} library \cite{rackauckas2017differentialequations}. 

\subsection{Compressible Euler equations}
In both 1D and 2D experiments, we investigate the compressible Euler equation governed by the ideal gas law. Let $\myu$ denote the vector of conservative variables. In 2D, 
\begin{align}
    \myu = \{\rho, \rho u_1, \rho u_2, \rho E\} \in \mathbb{R}^4
\end{align}
where $\rho, u_i, E$ denotes density, velocity in $i$th coordinate direction, specific total energy, respectively. Pressure $p$ is related to density and specific internal energy $e$ by the ideal gas law with constitutive relations
\begin{align}
    p = (\gamma - 1) \rho e, \quad E = e + \frac{1}{2} \sum_{i =1}^2 u_i^2,
\end{align}
where $e$ is the internal energy density and $\gamma$ is the heat capacity ratio which is set to $1.4$.

The compressible Euler equation in $d$ dimensions are
\begin{align}
    \frac{\partial \myu}{\partial t} + \sum_{m = 1}^d \frac{\partial \myf_m(\myu)}{\partial x_m} = \bm{0},
\end{align}
where $\myf_i$ is the convective flux in the $i$th coordinate direction. In 2D, 
\begin{align}
    \myf_1 = 
    \begin{bmatrix}
        \rho u_1 \\
        \rho u_1^2 + p \\
        \rho u_1u_2 \\
        u_1 (E + p)
    \end{bmatrix}, \quad 
    \myf_2 = 
    \begin{bmatrix}
        \rho u_2 \\
        \rho u_1u_2 \\
        \rho u_2^2 + p \\
        u_2 (E + p)
    \end{bmatrix}.
\end{align}
While the compressible Euler equations admit an infinite family of convex entropy functions \cite{harten1983symmetric}, the compressible Navier-Stokes equations possess a mathematical entropy inequality corresponding to only a single entropy function 
$S(\myu)$ and associated entropy potential $\psi_i(\myu)$
\begin{align}
    S(\myu) =-\rho s, \quad \psi_i(\myu) = \rho u_i
\end{align}
where $s = \log(\frac{p}{\rho^\gamma})$ denotes physical entropy under ideal gas law. We utilize this entropy for all numerical experiments.
\subsection{Density wave and high order accuracy}
We first verify high order accuracy of the proposed approach by performing a convergence test on the discretization using the Sv\"{a}rd thermal diffusive term with the Laplacian viscosity. We note that convergence test of schemes using SVV with the Laplacian viscosity model yield similar behavior. The initial condition for the 1D density wave test case for the compressible Euler equation is 
\begin{align*}
    (\rho, u, p) = (1 + 0.5\sin(x - t), 1, 1)
\end{align*}
We compute the $L^2$ errors at final time $T = 1.7$. Optimal rates of convergence are observed for polynomial degrees $N = 1, ... ,4$.
\begin{table}[H]
\centering
\small
\setlength{\tabcolsep}{4pt}
\renewcommand{\arraystretch}{0.9}

\begin{tabular}{c c c c c}
\toprule
$h$ & $N=1$ & Rate & $N=2$ & Rate \\
\midrule
$1/2$  & $6.514\times10^{-1}$ &      & $4.639\times10^{-2}$ &      \\
$1/4$  & $2.219\times10^{-1}$ & 1.554 & $8.181\times10^{-3}$ & 2.504 \\
$1/8$  & $6.034\times10^{-2}$ & 1.879 & $1.353\times10^{-3}$ & 2.596 \\
$1/16$ & $1.539\times10^{-2}$ & 1.971 & $1.890\times10^{-4}$ & 2.840 \\
$1/32$ & $3.865\times10^{-3}$ & 1.993 & $2.443\times10^{-5}$ & 2.952 \\
\bottomrule
\end{tabular}
\end{table}
\vspace{-1cm}
\begin{table}[H]
\centering
\small
\setlength{\tabcolsep}{4pt}
\renewcommand{\arraystretch}{0.9}
\begin{tabular}{c c c c c}
\toprule
$h$ & $N=3$ & Rate & $N=4$ & Rate \\
\midrule
$1/2$  & $5.375\times10^{-3}$ &      & $2.734\times10^{-4}$ &      \\
$1/4$  & $2.278\times10^{-4}$ & 4.560 & $1.368\times10^{-5}$ & 4.320 \\
$1/8$  & $1.346\times10^{-5}$ & 4.081 & $5.400\times10^{-7}$ & 4.663 \\
$1/16$ & $8.080\times10^{-7}$ & 4.058 & $1.845\times10^{-8}$ & 4.871 \\
$1/32$ & $5.015\times10^{-8}$ & 4.010 & $5.917\times10^{-10}$ & 4.963 \\
\bottomrule
\end{tabular}

\caption{Computed $L^2$ errors for the 1D density wave using Sv\"{a}rd multiple artificial viscosity-based entropy stable DG method.}
\label{tab:density_wave_1d}
\end{table}
\subsection{Receding flow problem}
The receding flow problem is a 1D Riemann problem that experiences an ``overheating" or ``false heating" issue that has been a long-standing open problems in the field of computational fluid dynamics \cite{liou2017recedingflow}. The false heating issue does not diminish even with refinement of mesh sizes and smaller time steps \cite{liou2013overheating}.

We assume uniform density and pressure, but a velocity gradient causes the flow to move away from the center region. The movement creates a rarefying region in the center where pressure, density, and temperature monotonically decreases to the center region. However, numerical methods consistently fail to accurately recover such behavior. The incorrect profile is attributed to the production of physical entropy in the center region which violates the 2nd law of thermodynamics \cite{liou2017recedingflow}. 

The Sv\"{a}rd thermal viscosity model adds additional numerical dissipation dependent on the thermal gradient. We expect that inclusion of the thermal diffusive term in the Sv\"{a}rd thermal viscosity model will improve the numerical solution without disrupting the profiles of other variables. 

Figure \ref{fig:receding flow} compares the temperature profile between numerical solution when utilizing Sv\"{a}rd thermal viscosity model and Laplacian viscosity model. We observe that the Sv\"{a}rd model reduces the large temperature spike but does not provide a fix to the false heating phenomena. Furthermore, Figure \ref{fig:receding density} showcases that the inclusion of thermal dissipation does not heavily impact the resolution of the density or pressure profile while improving the temperature profile. 
\begin{figure*}[H]
\[
    (\rho_0, p_0, u_0) = 
    \begin{cases} 
      (1.0, 0.4, -2.0) & x \le 0 \\
      (1.0, 0.4, 2.0)  & x > 0
   \end{cases}
   \]
   \caption*{1D Receding flow initial condition}
\end{figure*}

\begin{figure}[H]
    \hspace*{-2in}
    \centering
    \subfloat[\centering $N = 1, M = 200$]{{\includegraphics[width=6cm]{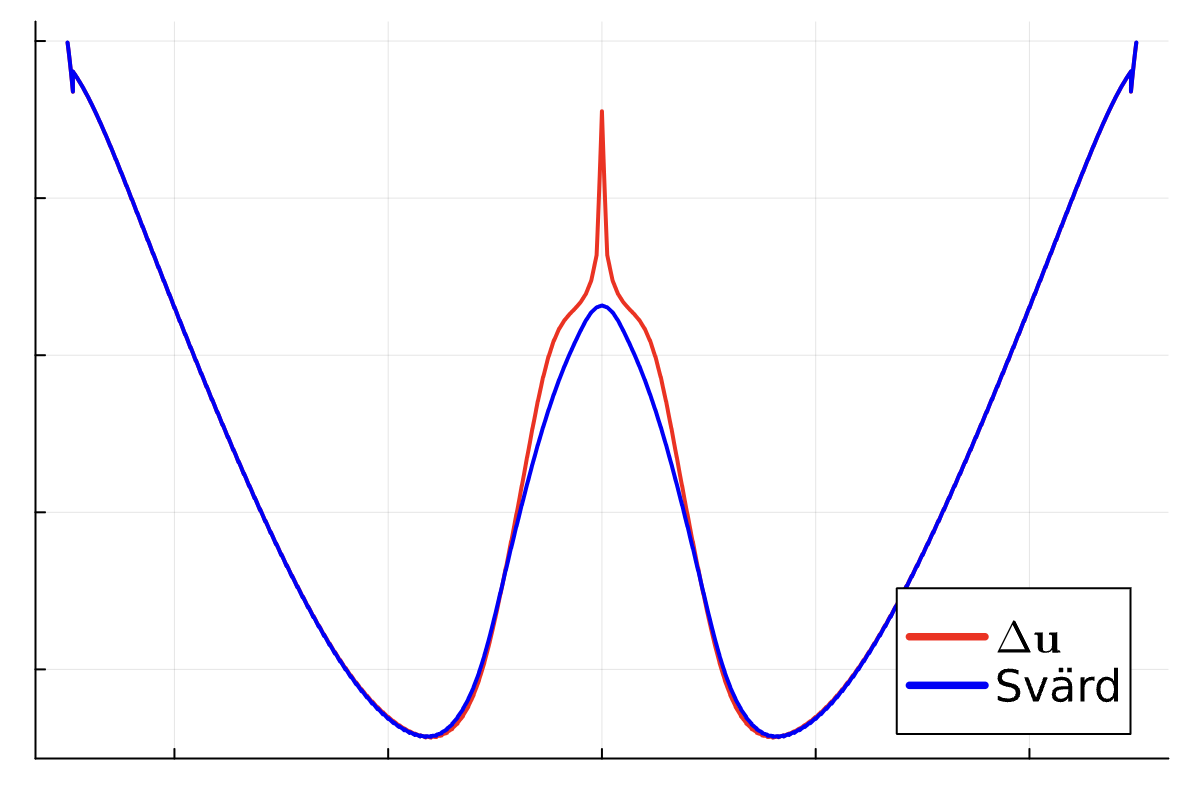} }}%
    \qquad
    \subfloat[\centering $N = 2, M = 130$]{{\includegraphics[width=6cm]{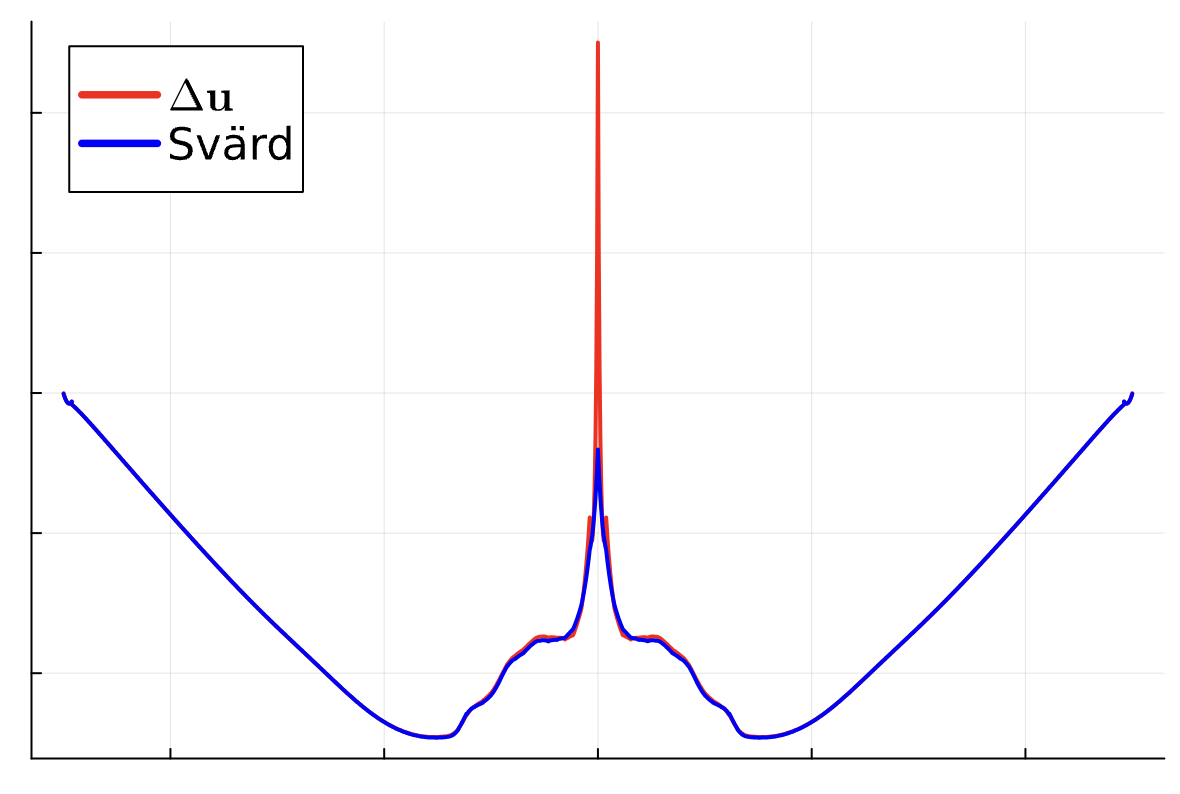} }}%
    \hspace*{-2in}
    \caption{Temperature profile of receding flow problem, $T = 0.18$}%
    \label{fig:receding flow}%
\end{figure}
\begin{figure}[H]
    \hspace*{-2in}
    \centering
    \subfloat[\centering Density]{{\includegraphics[width=6cm]{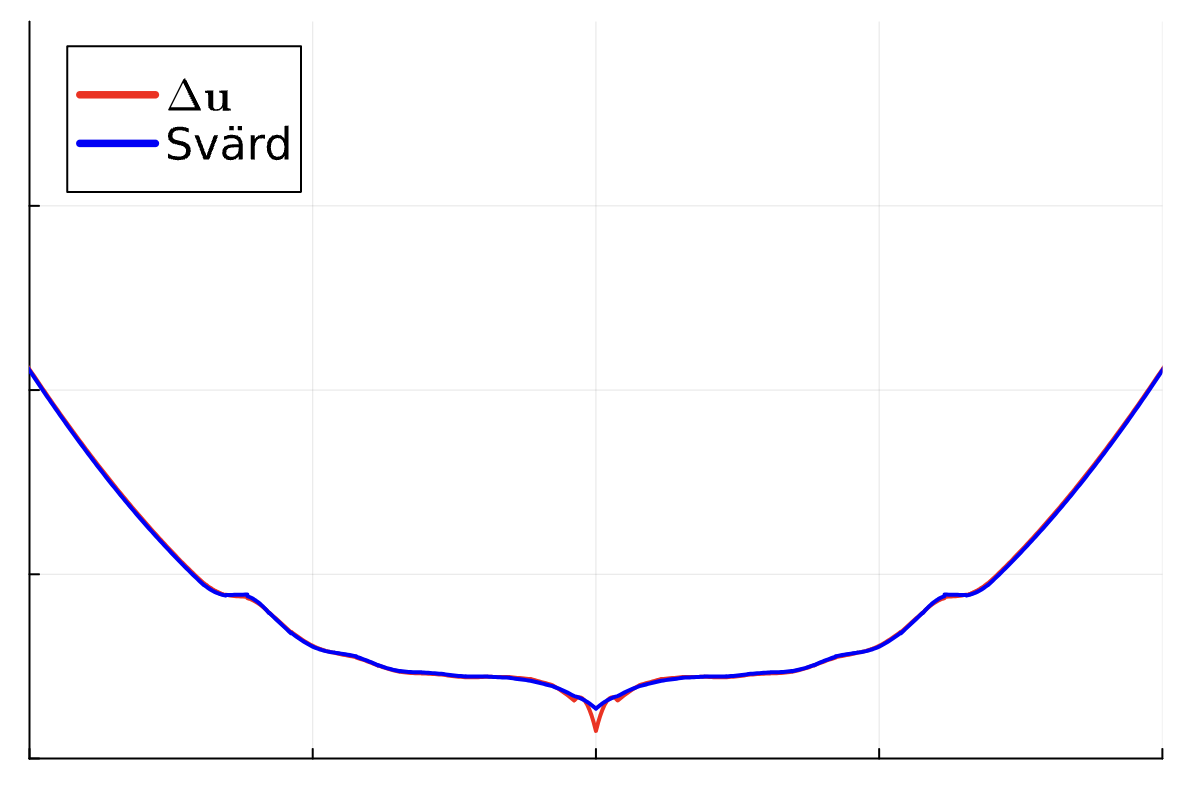} }}%
    \qquad
    \subfloat[\centering Pressure]{{\includegraphics[width=6cm]{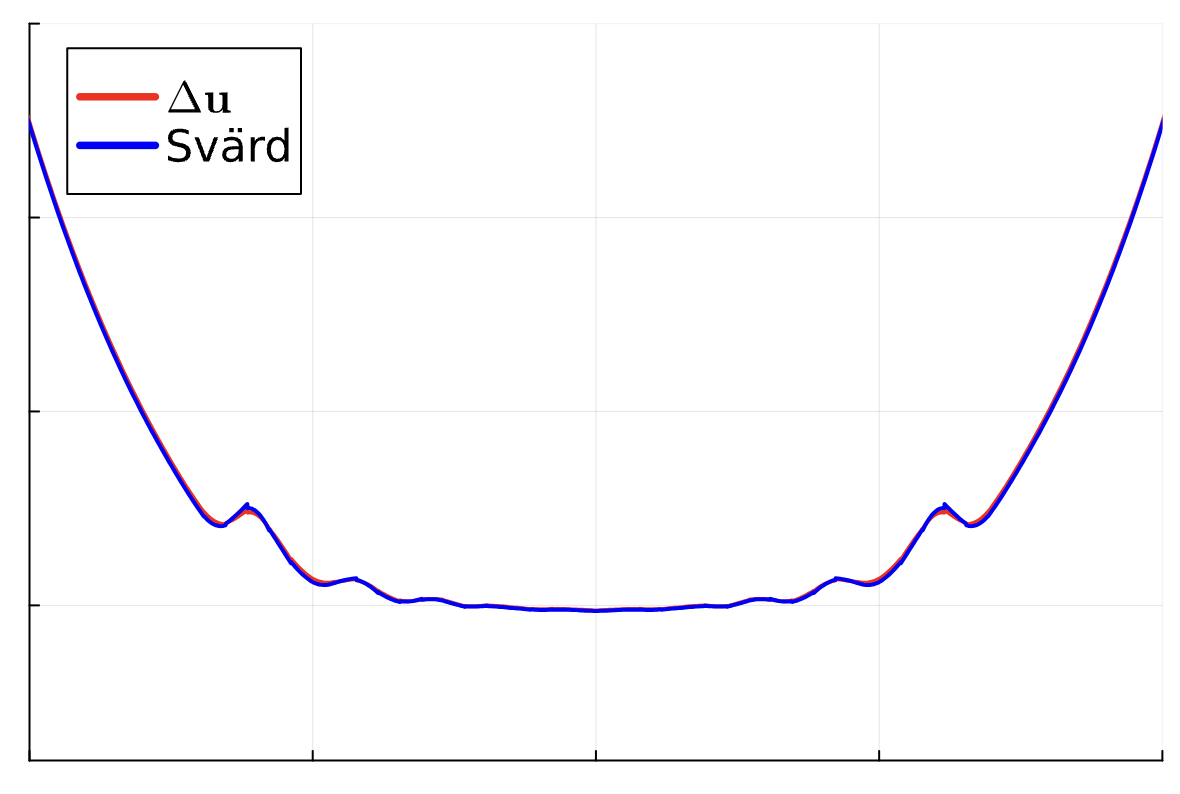} }}%
    \hspace*{-2in}
    \caption{Receding Flow $N = 2, M = 130$}%
    \label{fig:receding density}%
\end{figure}
\subsection{Kelvin Helmholtz instability}
We demonstrate the impact of incorporating the SVV model through long-time simulations of the Kelvin–Helmholtz instability (KHI). Problem setup can be found in \cite{chan2022entropyprojection}. As KHI exhibits shear stresses and features vorticular movements, we expect inclusion of SVV will better model the physics that dictate the phenomena, and the coefficient of SVV that is added will be non-negligible. 

Figure \ref{fig:KHI_Coef} shows the density profile of the long-time KHI and heat map of the SVV coefficient when using both SVV and Laplacian viscosity model. KHI is prone to positivity violations, which can cause simulations to crash due to negative density or pressure. In contrast, simulations using the proposed method remain free of such physical inconsistencies. We note that numerical solutions of the KHI depend strongly on the discretization and do not necessarily converge to a unique solution under mesh or polynomial degree refinement \cite{fjordholm2017construction}. Thus, we show this experiment primarily as a measure of robustness and as a comparison of the activation behavior of the different viscosity models. 

 The heat map in Figure \ref{fig:KHI_Coef} indicates that the SVV coefficient is not negligible, and the pattern exhibited in the heat map suggests that SVV is capable of detecting regions of high shear stress and under-resolved vortices. 
 
\begin{figure}[H]
    \hspace*{-2in}
    \centering
    \subfloat[\centering Density]{{\includegraphics[height=6cm]{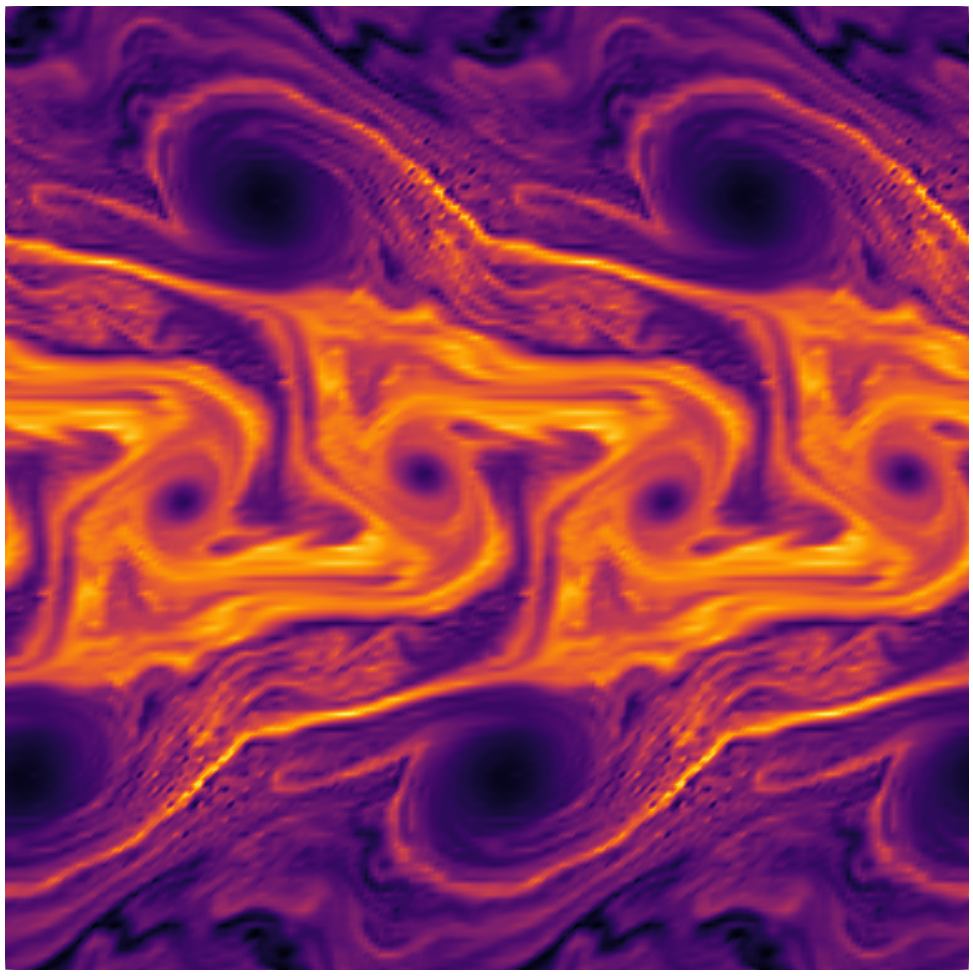} }}%
    \subfloat[\centering SVV coefficient $\epsilon^{SVV}_k(\myu_h)$]{{\includegraphics[height=6cm]{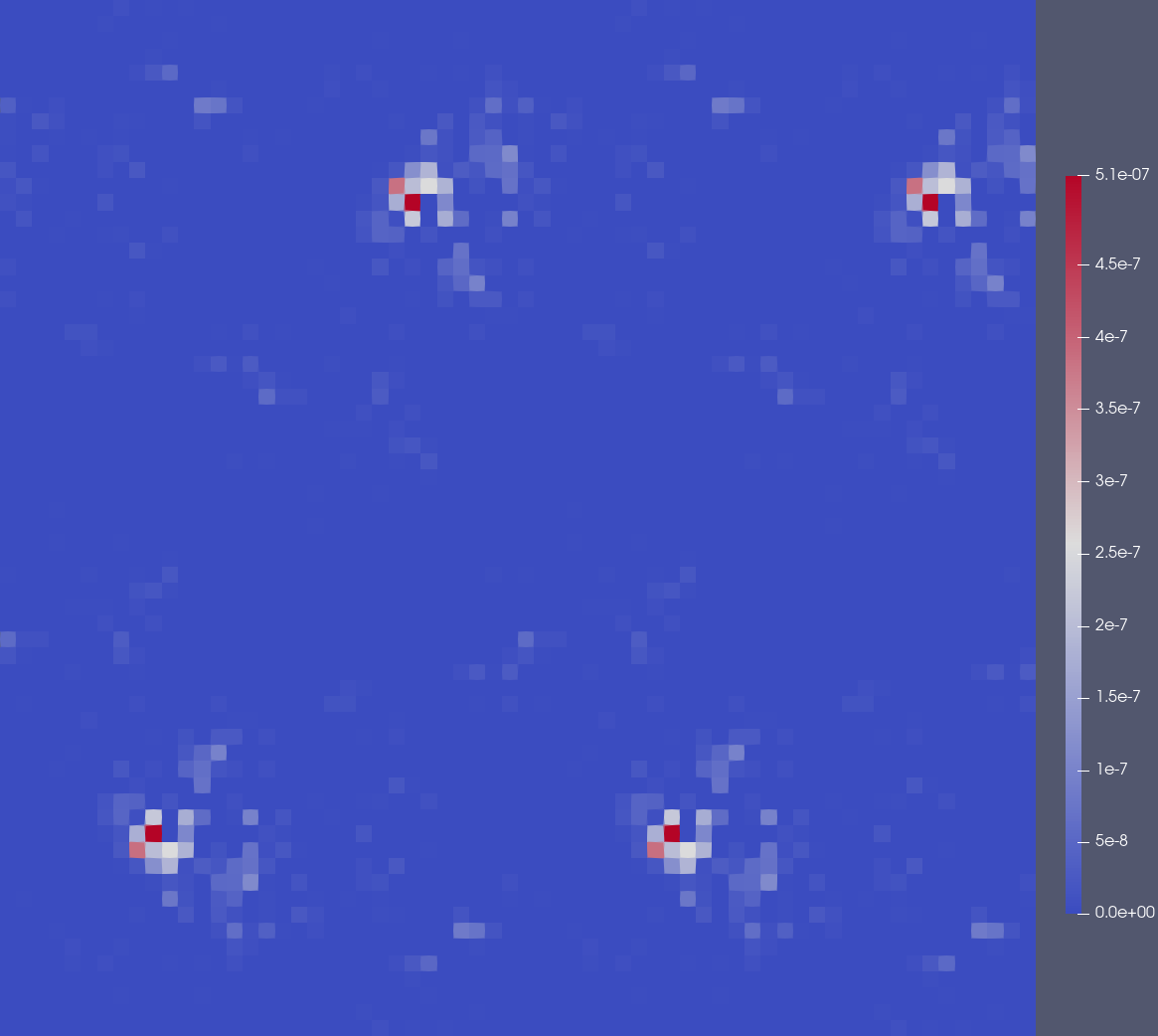} }}%
    \hspace*{-2in}
    \caption{Solution to Kelvin Helmholtz instability using SVV and Laplacian viscosity and SVV coefficient at final time $T = 25.0$, and using $N = 3$, $64 \times 64$ grid.}
    \label{fig:KHI_Coef}%
\end{figure}

\subsection{Riemann problem}
We conclude with a periodic version of a 2D Riemann problem adapted from \cite{kurganov2002solution} to show that the inclusion of SVV does not change the solution when shear stresses and vortices are not featured. The problem is posed on a $[-1, 1]^2$ domain with initial condition 
\begin{align*}
(\rho, u_1, u_2, p) = 
\begin{cases}
    (0.8, 0, 0, 1) & x < 0.5, y < 0.5 \\
    (1, \frac{3}{\sqrt{17}}, 0, 1) & x < 0.5, y > 0.5 \\
    (1, 0, \frac{3}{\sqrt{17}}, 1) & x > 0.5, y < 0.5 \\
    (\frac{17}{32}, 0, 0, 0.4) & x > 0.5, y > 0.5 \hspace{5pt}.
\end{cases} 
\end{align*}
\begin{figure}[H]
    \centering
    \subfloat[\centering Density]{{\includegraphics[width=3.8cm, height=3.8cm]{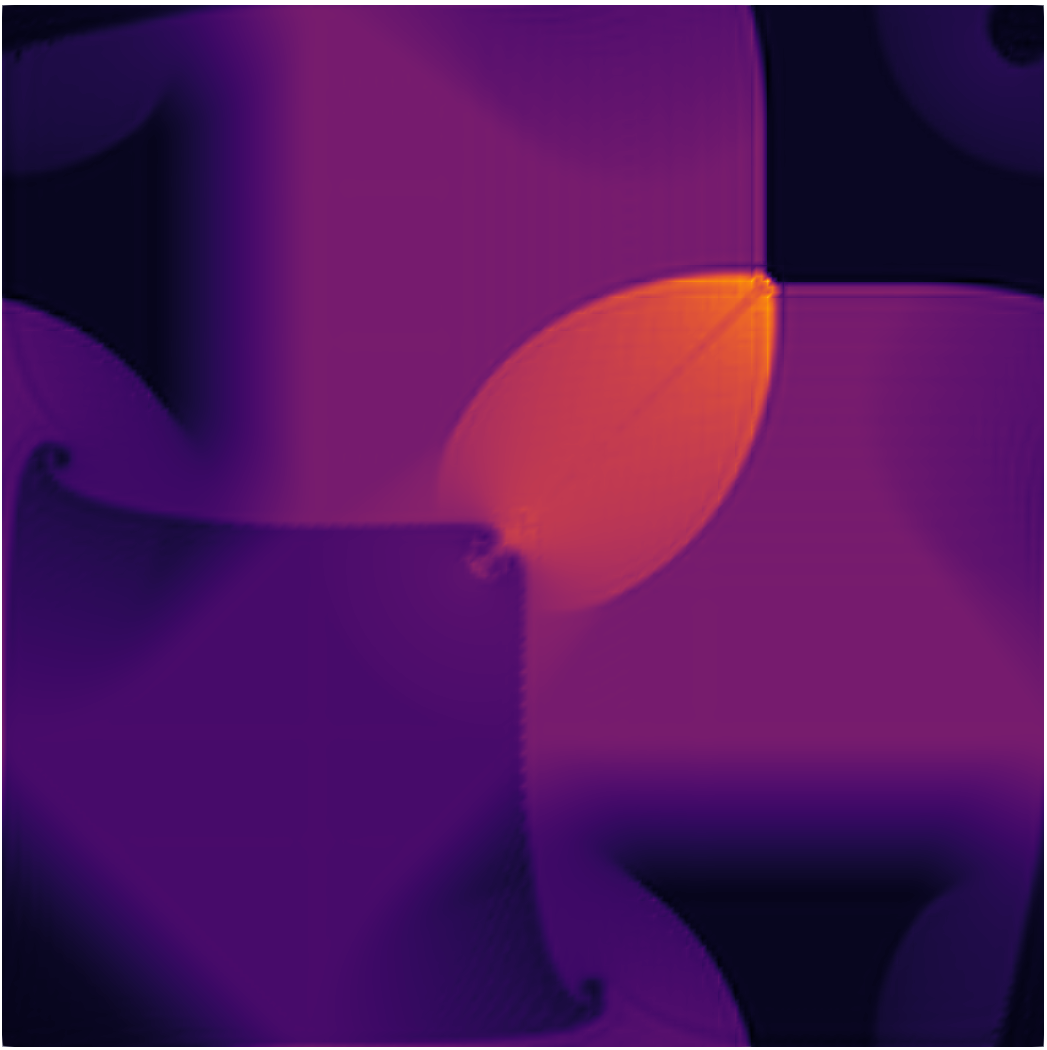} }}
    \subfloat[\centering SVV coefficient $\epsilon_{\Delta \myu_{SVV}}$]{{\includegraphics[width=3.8cm, height=3.8cm]{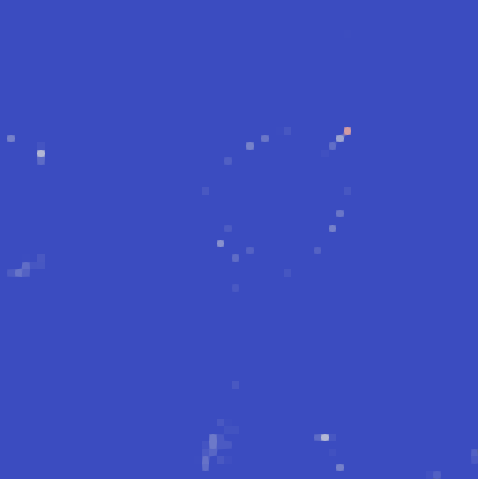} }}%
    \subfloat[\centering Laplacian coefficient $\epsilon_{\Delta \myu}$]{{\includegraphics[height=3.8cm]{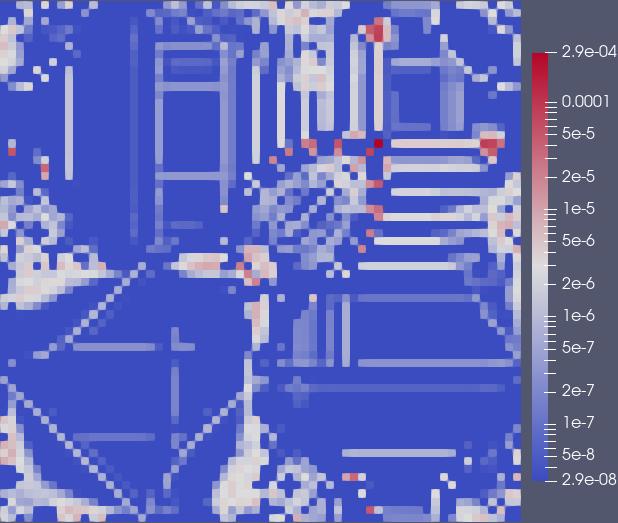} }}%
    \caption{$N=3, 64 \times 64, T = 0.14$, log scale. To highlight how small the SVV coefficient is relative to the Laplacian coefficient, we show the heat map of the log scale of the coefficients.}%
    \label{fig:riemann_problem}%
\end{figure}
Figure \ref{fig:riemann_problem} shows the density profile of 2D Riemann problem using SVV and Laplacian viscosity model and the log profile of the viscosity coefficients of each model. We note that incorporating only the Laplacian model does not change the solution profile. The log profile of the coefficients demonstrates that the SVV coefficient is near zero across the entire domain and is negligible relative to the Laplacian viscosity coefficient.
\section{Conclusions}
In this paper, we extend the ECAV framework introduced in $\cite{chan_av}$ by incorporating different viscosity models in an optimal parameter free fashion. We are able to suppress the presence of large spurious temperature spikes for the receding flow problem and demonstrate that SVV model is able to detect regions of turbulent, high shear stress regions while still enforcing an entropy condition. This procedure allows part of the dissipation required for entropy correction to be offloaded to other dissipative models. Using this framework, users can utilize a suitable viscosity model by deriving a
provably dissipative discretization of the model. The resulting scheme is capable of
targeting different physical phenomena by enforcing entropy stability.
\begin{acknowledgement}
The authors gratefully acknowledge support from National Science Foundation under the awards NSF-GRFP-DGE-2137420.
\end{acknowledgement}
\ethics{Competing Interests}{The authors declare no competing interests.}

\eject

\addcontentsline{toc}{section}{Appendix}

\bibliographystyle{spmpsci} 
\bibliography{biblio} 
\end{document}